\documentclass[reqno,10pt]{amsart}
\usepackage{amsfonts,amssymb,latexsym,epsfig,amsmath}
\usepackage[english]{babel}
\usepackage{amsthm}
\usepackage{color}

\newtheorem{theorem}{Theorem}[section]
\newtheorem{proposition}[theorem]{Proposition}
\newtheorem{corollary}[theorem]{Corollary}
\newtheorem{definition}[theorem]{Definition} 
\newtheorem{lemma}[theorem]{Lemma} 

\theoremstyle{remark}
\newtheorem{remark}[theorem]{Remark}
\numberwithin{equation}{section}

\newcommand{\R}{{\mathbb R}}

\newcommand{\I}{{\mathcal I}}

\newcommand{\bx}{{\bar x}}
\newcommand{\by}{{\bar y}}

\newcommand{\epsi}{\varepsilon}
\newcommand{\eps}{\varepsilon}
\newcommand{\ra}{\rightarrow}
\newcommand{\alp}{\alpha}

\newcommand{\SN}{\mathbb{S}^N}
\newcommand{\del}{\partial}
\newcommand{\ol}{\overline}

\DeclareMathOperator{\dist}{dist}

\def\domeg{\partial \Omega}
\def\omegb{\overline \Omega}
\def\pse{\psi_{\eps,\eta}}

\def\norm#1{|#1|}

\def\barre{|}
\def\lamb0{{\lambda_0}}

\begin{document}
\title[nonlocal Neumann  boundary value problems]
{On Neumann and oblique derivatives boundary conditions for
nonlocal elliptic equations} 
 
\author[G. Barles, C. A. Georgelin \&  E. R. Jakobsen]
{Guy Barles, Christine A. Georgelin \& Espen R. Jakobsen}

\address{Guy Barles \newline
Laboratoire de Math\'ematiques et Physique
Th\'eorique (UMR CNRS 7350)\newline
F\'ed\'eration Denis Poisson (FR CNRS 2964) \newline
 Universit\'e de
Tours, Parc de Grandmont \newline
 37200 Tours \newline 
 FRANCE 
 }
\email{barles@univ-tours.fr }
\urladdr{http://www.lmpt.univ-tours/$\sim$barles}

\address{Christine Georgelin \newline
Laboratoire de Math\'ematiques et Physique
Th\'eorique (UMR CNRS 7350)\newline
F\'ed\'eration Denis Poisson (FR CNRS 2964) \newline
 Universit\'e de
Tours, Parc de Grandmont \newline
 37200 Tours \newline 
 FRANCE  
 }
\email{christine.georgelin@univ-tours.fr }
\urladdr{http://www.lmpt.univ-tours/$\sim$georgeli}

\address{Espen R. Jakobsen \newline
    Department of Mathematical Sciences \newline 
    Norwegian University of Science and Technology \newline
    7491 Trondheim, Norway }
\email{erj@math.ntnu.no}
\urladdr{http://www.math.ntnu.no/$\sim$erj}

\date{\today}
\thanks {This work was supported by the Research Council
  of Norway   (NFR) through the project ``Integro-PDEs: Numerical methods,
  Analysis, and Applications to Finance''.}



\begin{abstract}
 Inspired by the penalization of the domain approach of Lions \&
 Sznitman, we give a sense to Neumann and oblique derivatives boundary
 value problems for nonlocal, possibly degenerate elliptic equations. Two different cases are considered: (i) homogeneous Neumann boundary conditions in convex,
  possibly non-smooth and unbounded domains, and (ii) general oblique
  derivatives boundary conditions in smooth, bounded, and possibly
  non-convex domains. In each case we give apropriate definitions of
  viscosity solutions and prove uniqueness of solutions of the
  corresponding 
  boundary value problems. We prove that these
    boundary value problems arise in the penalization of the domain
    limit from whole space problems and obtain as a corollary the
    existence of solutions of these problems.
\end{abstract}

\maketitle

\section{Introduction}
\label{sec:intro}

Inspired by the penalization of the domain approach of
Lions \& Sznitman \cite{pls} (see also \cite{LMS,MR}), we give a sense to
Neumann and oblique derivatives boundary value problems for nonlocal
degenerate elliptic partial integro-differential equations (PIDEs in short). 
  Because of the nonlocal nature of our PIDEs  posed in a domain
  $\Omega$, the boundary conditions have then to be imposed not only at
  the boundary $\del\Omega$, but possibly in all of the complement
  $\Omega^c$. At least boundary conditions must be imposed in the
  union of  the supports of the jump measures (see below).

To be more specific, we consider PIDEs of the form
\begin{equation}
   \label{E}
F(x,u,Du,D^2 u, \I [u](x)) =\ 0
  \quad \text{in}  \quad\Omega,
\end{equation}
with {\em extended} Neumann/oblique derivatives boundary conditions
\begin{equation}
Du(x)\cdot \gamma(x)=\ g(x) \quad\text{in}\quad \Omega^c.\label{BC}
\end{equation}
Here $\Omega$ is a domain in $\R^N$ and $F$ is a
real-valued, continuous function defined on $\R^N \times \R \times
\R^N \times \SN \times \R$, where $\SN$ is the space of  $N\times N$ symmetric
matrices. We assume that $F$ is degenerate elliptic which, in this
nonlocal setting, means that for any $x \in \R^N$, $u \in \R$,  
$p \in \R^N$, $M, N \in \SN$ and $l_1, l_2 \in \R$, 
\begin{equation}\label{deg-ell}
F(x,u,p,M,l_1) \leq F(x,u,p,N,l_2)\quad\hbox{when}\quad M\geq N,\ l_1
\geq l_2,
\end{equation}
where $M\geq N$ has to be understood in the sense of the usual partial ordering on symmetric matrices.
 Our assumptions cover the cases of general linear (see example below)
 and nonlinear equations and, in particular, Bellman-Isaacs equations of control and game theory.

The equation is nonlocal because of its dependence on the nonlocal
operator $\I [u]$ which we assume to be of L\'evy-Ito type. For any
smooth bounded function $\phi$ and for any $x\in \R^N$, 
\begin{equation}
\label{defLevyIto}
\I[\phi](x)=\int_{\R^N}\left [\phi(x+j(x,z))-\phi(x)-D\phi(x)\cdot j(x,z) 1_B(z)\right ]\, d\mu(z) ,
\end{equation}
where $B\subset\R^N$ is the unit ball, $1_B$ the indicator function of
$B$, $\mu$ -- the L\'evy measure -- is a positive Radon measure on
$\R^M\setminus\{0\}$ and $j : \R^N \times \R^M \to \R^N$ is a function which is $\mu$-measurable in $z$ and continuous in $x$ for $\mu$-a.e. $z$, and there exists a constant $c(j)>0$ such that
\begin{align}
\label{Levy}
 \int_{\R^N} |z|^2\wedge 1\ d\mu(z)<\infty\quad\text{and}\quad
 |j(x,z)|\leq c(j)|z|\quad\text{for any}\; |z|<1,\ x\in \R^N.
\end{align}
A Taylor expansion shows that $\I [\phi]$ is well-defined under
\eqref{Levy}. {\em We will assume that \eqref{Levy}
  holds throughout this paper.}  The operator $\I $ is the generator
of a stochastic jump process which solves a stochastic differential
equation involving a general jump term/Poisson random measure,
cf. \cite{GS:Book,pha}.
    Included are generators of all  pure jump Levy processes
    \cite{A:Book} as well most Levy models arising in Finance \cite{CT:Book}.

A typical example of Equation~(\ref{E}) is the linear equation
\begin{equation}\label{example}
 a(-\Delta)^{\frac\alp2} u - {\rm Tr}(A(x)D^2 u) -b(x)\cdot Du + \lambda (x) u = f(x)  \quad \text{in}  \quad\Omega,
\end{equation}
where $a\in \R$, $A, b, \lambda, f$  are continuous functions on $\omegb$,
taking values respectively in $\SN$, $\R^N$, $(0,+\infty)$ and
$\R$, and, for $\alp\in(0,2)$, $\I=-(-\Delta)^{\frac\alp2} $ is
the Fractional Laplacian in $\R^N$, the generator of the symmetric
$\alp$-stable processes. It is defined e.g. by \eqref{defLevyIto} with
$j(x,z)\equiv z$ and $\displaystyle d\mu(z)= c_\alp
\frac{dz}{|z|^{N+\alpha}}$ for some constant $c_\alp>0$.
To satisfy the (degenerate) ellipticity condition (\ref{deg-ell}), we
must impose that both $a\geq0$ and $A \geq 0$ in $\omegb$.

It is the definition of $\I$  (cf. (\ref{defLevyIto})) that
requires $u$ to be defined in all of $\R^N$, and hence that the
boundary condition must be posed in all of $\Omega^c$.

For the boundary (or exterior) condition \eqref{BC}, we assume that 

\smallskip
\begin{itemize}
\item[(BC1)] The functions $\gamma : \R^N \to \R^N$ and $g: \R^N \to
  \R$ are bounded Lipschitz continuous functions, and there
  exists $\nu >0$ such that $\gamma (x) \cdot n (x) \geq \nu$ for any
  $x \in \domeg$.
\item[]
\item [(BC2)] For any $x \in \omegb^c$, $\tau_x := \inf_{t>0}\{ X_x(t) \in \omegb\}
  <+\infty$, where $X_x(\cdot)$ solves
\begin{equation}\label{ode}
X_x(0) = x \qquad\text{and} \qquad \dot X_x(t)=- \gamma (X_x(t))\quad\text{for}\quad t>0.
\end{equation}
\end{itemize}

Assumption (BC1) is sufficient for \eqref{BC} to really play the role
of a boundary condition in the case of local equations.  Assumption
(BC2) states that integral curves of the vector field $-\gamma$
starting from any point $x\in\Omega^c$, will reach the boundary
$\del\Omega$ in finite time. This is a natural condition for a Neumann
type boundary condition, and it is very closely related to the idea of the
``penalization of the domain'' method of  Lions \& Sznitman
\cite{pls} (see also \cite{LMS,MR}). This method is based on the
observation that in the limit $\kappa\ra0$, the vector field $-\frac 1 \kappa
\gamma$ instantaneously returns the underlying stochastic process to
$\omegb$  after an outside jump, and this is where \eqref{ode} plays a role. We refer to
Section~\ref{asymp} for more details in this direction.

As in Lions \& Snizman \cite{pls}, we use the notion of viscosity
solutions. For nonlocal equations posed in full space, we refer to
\cite{BI} (see also \cite{BBP,JK,pha}) and references therein for an
account of this theory.  A nonlocal Dirichlet problem was considered
in \cite{BCI}, where boundary conditions are given in all of
$\Omega^c$ in an analogous way as in this paper. Here we consider two
different cases: (i) a very general class of equations with
homogeneous Neumann boundary conditions posed in convex, possibly
non-smooth and unbounded domains, and (ii) a less general class of
equations with general oblique derivatives boundary conditions in
smooth, bounded, and possibly non-convex domains.  In each case we
give appropriate definitions of viscosity solutions and prove
uniqueness theorems for the corresponding boundary value
problems.  
 Here we want to point out that the extended oblique
derivative condition (\ref{BC}) influences the behavior of the solutions at infinity and therefore
interferes in the conditions which are needed to have a well-defined
nonlocal operator; we discuss this point at the end of
Section~\ref{prelim}. We also show that our formulation follows 
  from a sequence of problems posed in the whole space obatined from
 the
  penalization of the domain method in a similar way as 
in \cite{pls}. As a consequence we also get some existence results for
our problems. 

In a related paper \cite{bcgj}, the authors
along with E. Chasseigne investigate four different ways of
understanding homogeneous Neumann boundary conditions for Levy-type
nonlocal equations posed in the half space $\Omega= \R^{N-1}\times
\R^+$. Here, the simple geometry of the domain allows
formulations where the nonlocal operators and equations are restricted
to $\overline \Omega$. One of the cases, the one involving normal projection of
outside jumps, is linked to the present work. 
The restriction to $\overline \Omega$ formulation of \cite{bcgj} is technically
more difficult to work with than the very natural full space
formulation we use in this paper. This also explains why we now obtain
much more general results for the normal/oblique projection type of
models. Where we in \cite{bcgj}  
considered linear, non-degenerate problems with a restricted class of
Levy operators on a simple domain, we can now treat very general
equations, Levy operators, and domains, and inhomogeneous and even oblique boundary conditions.

In some cases covered in this paper, a probabilistic
 description of the reflection problems based on stochastic
 differential equations can be found in \cite{MG96}. 
Elsewhere in the literature similar problems have been
investigated for L\'evy operators where the measure
$\mu_x$ forces the underlaying process to stay in the domain either by
a ``smooth'' restriction of its support or by ``killing'' all jumps 
leaving $\Omega$. In these cases a Neumann boundary condition can be
imposed only at the boundary $\del\Omega$, just as for local
problems. The first type of problems is considered in \cite{MR}, see
also the book \cite{GM}, and the killing approach is linked to the
$\alpha$-censored process \cite{BBC} and the regional
fractional Laplacian \cite{guan,guanmabd,guanmasym}. 

When the underlying process is a symmetric $\alp$-stable
processes (a subordinated Brownian motion), the above mentioned
approaches follows after a  ``reflection'' on the boundary: The 
processes can be constructed from a Brownian motion by
first subordinating it and then reflecting it. Another possible way
to construct a  ``reflected'' process it to first reflect the Brownian
motion and then  subordinate the reflected process. This approach is related to
 Dirichlet-Neumann operator, and it have been described
 e.g. by Hsu \cite{Hsu} using probabilistic methods and by Caffarelli
 and Silvestre \cite{CC} by analytic PIDE methods. Especially the ideas of
 \cite{CC} have been used by many authors since.

In all these approaches,  as well as in \cite{bcgj}, the
  non-local operator is no longer the orginal $\R^N$-operator. In fact
  the operator and hence also the equation will depend on the domain,
  and different domains $\Omega$ yield different operators and
  different equations inside $\Omega$. 

Our paper is organized as follows: Section \ref{prelim} is devoted to
a key technical lemma which allows us to control the solutions
outside the domain $\Omega$.
 We also discuss the connections between the extended oblique derivative condition (\ref{BC}), the behavior of the solutions at infinity, and the conditions which are needed to have a well-defined nonlocal operator. In Section 
\ref{sec:results}, we will focus on convex possibly non-smooth and
unbounded domains but restrict ourselves to homogeneous boundary
conditions. We define the concept of viscosity solution and proof a
comparison theorem. The key argument here is to obtain by convexity a
contraction property that force maximum points of the test function to
be in $\bar \Omega$.  After this, the proof can be concluded in the
standard (full space) way. In Section~\ref{gen-obl}, we prove a
comparison theorem in the case of general
oblique derivative conditions and smooth bounded possibly nonconvex
domains. The proof uses the complicated test function constructed by
G. Barles in \cite{ba2}, along with the technical lemma  of
Section~\ref{prelim}.  Finally, Section~\ref{asymp} is devoted to the analysis
of an asymptotic result -- the penalization of the domain method
introduced by Lions and Sznitman. We prove that the above boundary
value problems arise in the penalization of the domain limit of whole
space problems and obtain as a corollary existence results for our
Neumann problems.

\section{Preliminary Results}\label{prelim}

In this section we state and prove two lemmas which play key roles
in the proofs in the next sections. We recall from (BC2)
  that $\tau_y := \inf_{t>0}\{ X_y(t) \in \omegb\}$ for $y \in  \omegb^c$.

\begin{lemma}\label{backtoboundary} Assume (BC1) and (BC2).
\medskip

\noindent (a) If $u$ is a locally bounded, usc function satisfying
$Du(x)\cdot \gamma (x) 
\le g(x) $ in $\omegb^c$ in the viscosity sense, then, for any $y \in \omegb^c$ and for any $t \leq \tau_y$, we have
\begin{equation}
\label{maxatbdr}
u(y) \le \int_0^t g(X_y(s))ds + u(X_y(t)) .
\end{equation}
\smallskip
(b) If $v$ is a locally bounded, lsc function satisfying $Dv(x)\cdot \gamma (x)
\ge g(x) $ in $\omegb^c$ in the viscosity sense, then, for any $y \in \omegb^c$ and for any $t \leq \tau_y$, we have
\begin{equation}
\label{minatbdr}
v(y) \ge \int_0^t g(X_y(s))ds + v(X_y(t)) .
\end{equation}
\end{lemma}

 \begin{figure}[h!]
 \includegraphics[scale=1.00]{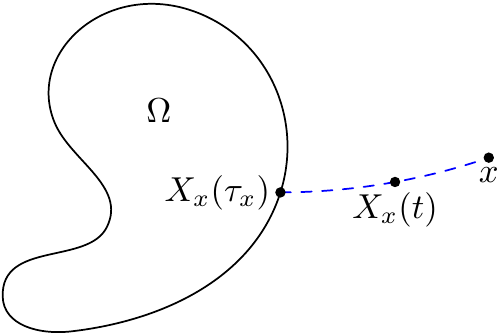}
 \caption{The curve $X_x(t)$ of (BC1) and the hitting time $\tau_x$.}
 \end{figure}

\begin{proof} The proof is inspired by an argument of  G. Barles, S. Mirrahimi, B. Perthame and P.E. Souganidis \cite {bmps}. We only prove (a) since the proof of (b) is similar.
Let $t\in(0,\tau_y]$ and define
$$w(s):=u(X_y(t-s))\quad\text{for }s \geq 0 .$$ 
If we can prove that the usc function $w$ is a subsolution of
\begin{align}
\label{w_eq}
\frac{dw}{ds}(s)=g(X_y(t-s))\quad\text{for  } s>0,
\end{align}
then by the comparison principle
$$ w(s) \leq w(0) + \int_0^s g(X_y(t-\tau))d\tau\quad\text{for  } s>0,$$
since the right-hand side is the $C^1$-solution of \eqref{w_eq} with
same initial data $w(0)$. Now \eqref{maxatbdr} follows by choosing
$s=t$. 

To prove that $w$ is a supersolution of \eqref{w_eq}, we take any
smooth test-function $\varphi$ and any point $\bar s >0 $ such that
$w-\varphi$ has a local strict maximum point 
at $\bar s$. Note that $X_y(t-\bar s)\in \omegb^c$ since $\bar s >0 $. Next we introduce the functions
\begin{align*}
\phi(x,s)=\frac{|X_y(t-s)-x|^2}{\eps^2}+\varphi(s)\qquad\text{and}\qquad \psi(x,s)=u(x)-\phi(x,s).
\end{align*}
For $\eps>0$ small enough, classical arguments show that $\psi$ has a
maximum point near $(X_y(t-\bar s),\bar s)$ (depending on $\eps$) that we also call $(x,s)$. Moreover
\begin{equation}\label{kp1}
\frac{|X_y(t-s)-x|^2}{\eps^2}\ra0\quad\text{and}\quad
s \ra\bar s\qquad \text{as}\qquad \eps\ra0.
\end{equation}
Therefore $x \in \omegb^c$ for $\eps$ small enough, and since $u$ is a
subsolution of \eqref{BC} and $u-\phi(\cdot,t)$ has a local maximum at $x$,
$$-\frac{2(X_y(t-s)-x)}{\eps^2}\cdot \gamma(x)\leq g(x).$$
Since $\tau \mapsto\psi(x,\tau)$ is a $C^1$-function having a
local maximum at $\tau=s$, we also have
$$\frac{2(X_y(t-s)-x)}{\eps^2}\cdot \dot X_y(t-s)-\dot
\varphi(s)=\frac{\del\psi}{\del t}(x,s) =0, $$
and we can conclude using $\dot X_y(t-s)=-\gamma (X_y(t-s))$ and Lipschitz continuity of $\gamma$ that
\begin{align}
\dot\varphi(s) = &- \frac{2(X_y(t-s)-x)}{\eps^2}\cdot \gamma(X_y(t-s)) \nonumber \\
\leq & - \frac{2(X_y(t-s)-x)}{\eps^2}\cdot \gamma(x) - \frac{2(X_y(t-s)-x)}{\eps^2}\cdot \left(\gamma(X_y(t-s))-\gamma (x)\right)\nonumber \\
\leq & \ \ g(x) + O\Big(\frac{|X_y(t-s)-x|^2}{\eps^2}\Big)\; . \label{urg}
\end{align}
In view of \eqref{kp1}, we can send $\eps\ra0$ to find that
$\dot\varphi(\bar s)\leq g(X_y(t-\bar s))$.
\end{proof}

To allow for convex domains with corners and $\gamma=n$, we need to relax the
Lipschitz assumption on $\gamma$ in (BC1) and impose only a one-sided
Lipschitz condition
\smallskip

\begin{itemize}
\item[(BC1')] The functions $\gamma : \omegb^c \to \R^N$ and $g :
  \omegb^c \to \R$ are bounded continuous functions,
\begin{equation}\label{prop-imp}
 (\gamma (x) - \gamma (y))\cdot (x-y) \geq - K |x-y|^2\; ,
\end{equation}
for some constant $K$ and for all $x,y \in \omegb^c$, and there
  exists $\nu >0$ such that $\gamma (x) \cdot n (x) \geq \nu$ for any
  $x \in \domeg$.
\end{itemize}
\medskip

We state a slight generalization of Lemma \ref{backtoboundary}.

\begin{lemma} \label{rem22} The results of Lemma \ref{backtoboundary}
  remain valid if we replace (BC1) by (BC1'). 
\end{lemma}

\begin{proof} We first remark that (BC1') ensures the existence and
  uniqueness of the trajectory $X_y$ as long as it remains in
  $\omegb^c$, existence follows from Peano's Theorem while
  \eqref{prop-imp} provides the uniqueness. Therefore the trajectory
  exists on $[0,\tau_y)$ and can be extended by continuity to
  $t=\tau_y$. 

To complete the proof, we let $t < \tau_y$ and redo the proof of Lemma
\ref{backtoboundary}. The computations are exactly the same, e.g. to
obtain  \eqref{urg}, (BC1') is sufficient. To extend the result to
$t=\tau_y$, we just remark that 
$$\limsup_{t\uparrow \tau_y} u(X_y(t)) \leq u(X_y(\tau_y))\; ,$$
by the upper-semicontinuity of $u$.
\end{proof}

We conclude this section by an important discussion on the consequences of 
Lemma~\ref{backtoboundary}. If $u$ is a solution of (\ref{E})-(\ref{BC}), then $Du(x)\cdot \gamma (x)  = g(x) $ in $\omegb^c$ and we have, for all $y\in \omegb^c$
$$u(y) = \int_0^{\tau_y} g(X_y(s))ds + u(X_y(\tau_y)) .$$
If $g\equiv 0$ on $\omegb^c$, then $u$ can be bounded under suitable assumptions on the
(other) data. This is the case we face in
Section~\ref{sec:results} below for convex domains and extended homogeneous Neumann boundary conditions.
But if $g$ is not identically $0$ on $\omegb^c$, then $u$ can be unbounded and its growth is governed by the properties of $g$ and $\tau_y$.

The behavior of $u$ at infinity is important in our framework to
insure that the nonlocal terms are well-defined: We need some
integrability property like e.g., for any $x\in \omegb$ and $\delta>0$,
\begin{equation}\label{integ}
\int_{|z|\ge \delta} |u(x+j(x,z))|\, d\mu(z) <+\infty.
\end{equation}
Such condition now connects the assumptions we have to place on
$\tau_x$, $g$, $j$ and $\mu$.
 
To fix ideas, we are going to assume in Section~\ref{sec:5.2} that:
\begin{itemize}
\item [(BC3)] Either the function $g$ has a compact support in $\R^N$
  or there exists $\tilde c>0$ such that for $\tau_x$ from (BC2),
\begin{equation}
\tau_x \leq \tilde c (1+|x|)\qquad \hbox{and}\qquad \sup_{x\in \ol\Omega}\int_{|z|\ge \delta} |j(x,z)|\ d\mu(z)<\infty\ \hbox{for any $\delta >0$}\; .
\end{equation}
\end{itemize}

We briefly comment on this assumption. When $g$ has compact support,
the solutions are expected to be bounded by Lemma \ref{backtoboundary}
and no additional assumption on $j$ and $\mu$ is needed.  On the
contrary, if e.g. $g\equiv 1$, then the integral of $g$ in
(\ref{maxatbdr}) suggests that $u$ and $\tau_x$ have the same growth
and (BC3) imposes a linear growth. Next one has to impose suitable
hypothesis on $j$ and $\mu$ to satisfy (\ref{integ}). This
is obtained through the second part of (BC3) on the
$\mu$-integrability of $j$ away from $0$.

\section{The homogeneous Neumann condition in convex non-smooth domains}
\label{sec:results}

In this section we consider the homogeneous Neumann problem, namely
equation \eqref{E} and boundary condition \eqref{BC} with $g\equiv0$
and $\gamma=n$, the unit outward normal vector field in $\Omega^c$ (see below)
\begin{equation}
Du(x)\cdot n(x)=\ 0 \quad\text{in}\quad \Omega^c,\label{BC1}
\end{equation}
in the case when $\Omega$ is a convex, possibly unbounded and non-smooth domain.

At $x \in \domeg$, the set of outward normals $N_\Omega(x)$ can be
defined as
$$N_\Omega(x):=\Big\{n\in \R^N: |n|=1, n\cdot(x-y)\geq0 \text{
for all }y\in \omegb\Big\}.$$
This set is a singleton at any point where $\del\Omega$ is $C^1$ and  part of a
convex cone where $\del\Omega$ has a corner. Let $\bar d$ be the distance function to $\omegb$, and note that $\bar
d\equiv 0$ in $\bar \Omega$ and $\bar d>0$ in $\omegb^c$. Moreover, $\bar
d$ is convex and belongs to $C^0(\R^N)\cap C^1(\omegb^c)$ since
$\omegb$ is a closed convex subset of $\R^N$. In $\bar\Omega^c$, we
now {\em define} the outward unit normal vector $n$ in the only
natural way by setting $n=D\bar d$. Note that the two definitions
are consistent in the sense that
$$N_\Omega(x)=\Big\{n\in\R^N: n=\lim_{k\ra\infty}D\bar
d(x_k)
\text{ for some $x_k\ra x$}\Big\}\quad\text{for all}\quad x\in\del\Omega,$$ 
and since $\bar d$ is convex, the function $n$ 
satisfies (BC1') with $K=0$.

To define the concept of viscosity solutions for this problem, we need
the operators $\I_\delta$, $\I^\delta$, and $\mathcal{F}
$ defined as follows
\begin{align*}
&\I_\delta[\phi](x)=\int_{|z|<\delta}\phi(x+j(x,z))-\phi(x)-D\phi(x)\cdot
j(x,z) 
1_B(z)\, d\mu(z) ,\\
&\I^\delta[u](x)=\int_{|z|\ge
  \delta}u(x+j(x,z))-u(x)-D\phi(x)\cdot j(x,z) 1_B(z)\, d\mu(z),\\
&\mathcal{F}
[u,\phi]
(x)= F(x,u(x),D\phi(x),D^2\phi(x), \I_\delta[\phi](x)+\I^\delta[u](x)).
\end{align*}

 Under assumption \eqref{Levy}, $\I_\delta[\phi]$ is well-defined for
 $\phi\in C^2$. For $\I^\delta[u]$ we need some integrability
 condition a la (\ref{integ}), see the discussion at the
 end of Section~\ref{prelim}. In this section $g\equiv
 0$ on $\omegb^c$, so (\ref{integ}) will be automatically satisfied whenever 
$u$ is bounded.

\begin{definition}
\label{def_vsol}
(i) A locally bounded, usc function $u:\R^N\to \R$ is a {\em viscosity subsolution} of
\eqref{E}--\eqref{BC} if it satisfies \eqref{integ}, and for any test
function $\phi\in C^2(\R^N)$ and for any maximum point $x_0\in\R^N$ of
$u-\phi$ in $B_{c(j)\delta}(x_0)$ where $c(j)$ is defined in \eqref{Levy}, we have  
\begin{align*}
\begin{cases}
\mathcal{F}
[u,\phi]
(x_0) \le 0& \hbox{if}\quad
x_0\in \Omega,\\[0.2cm]
\displaystyle\min\bigg(\mathcal{F}
[u,\phi]
(x_0), \inf_{n\in N_\Omega(x_0)}D\phi(x_0)\cdot n\bigg)\le 0 &\hbox{if}\quad
x_0\in \partial \Omega,\\[0.2cm]
D\phi(x_0)\cdot n(x_0)\le 0 & \hbox{if}\quad x_0\in \omegb^c
\end{cases}
\end{align*}
(ii) A locally bounded, lsc function $v:\R^N\to \R$ is a {\em viscosity supersolution} of
\eqref{E}--\eqref{BC} if it satisfies \eqref{integ}, and for any test function $\phi\in C^2(\R^N)$ and for any minimum point $x_0\in \R^N$ of the function $u-\phi$ in $B_{c(j)\delta}(x_0)$ where $c(j)$ is defined in \eqref{Levy}, we have  
\begin{align*}
\begin{cases}
\mathcal{F}
[u,\phi]
(x_0) \ge 0& \hbox{if}\quad
x_0\in \Omega,\\[0.2cm]
\displaystyle \max\bigg(\mathcal{F}
[u,\phi]
(x_0),\sup_{n\in N_\Omega(x_0)} D\phi(x_0)\cdot n\bigg)\ge 0 &\hbox{if}\quad
x_0\in \partial \Omega,\\[0.2cm]
D\phi(x_0)\cdot n(x_0)\ge 0 & \hbox{if}\quad x_0\in \omegb^c
\end{cases}
\end{align*}
(iii)  A {\em viscosity solution} $u$ of \eqref{E}--\eqref{BC} is a
locally bounded function whose upper and lower semicontinuous envelopes are
respectively sub- and supersolution of the problem.
\end{definition}

This definition is a natural extension of the definition given in
\cite{BCI} to the Neumann type boundary value problem.

\begin{remark}
Two useful equivalent definitions can be given: (1) We can replace
$\I^\delta[u]$ by $\I^\delta[\phi]$ in the above definition if
local maximum/minimum points are replaced by global ones. (2) In the
subsolution definition, $(D\phi(x_0),D^2\phi(x_0))$ can be replaced by
elements $(p,X)$ in the so-called super-jet $J^{+}u(x_0)$ if
$X\leq D^2\phi(x_0)$. In the definition of supersolutions, you can
similarly use $(q,Y)\in J^{-}u(x_0)$ if $Y\geq D^2\phi(x_0)$. The
second definition is useful for comparison proofs, and the proofs of
these claims easily follow from the arguments for similar results in
\cite{BI}.
\end{remark}

We now state the assumptions -- remarking that the assumptions on $F$
will be as general as for the whole space case $\Omega=\R^N$
without boundary conditions. For convenience we use the assumptions of
\cite{BI}, but see Remark \ref{alt_assumptions} below for more general assumptions.
For the nonlocal part we assume that  
\smallskip
\begin{itemize}
\item[(A1)]
Assumption \eqref{Levy} holds, and there is a constant $\bar c>0$ such that for all $x,y\in\R^N$,
\begin{align*}
&
\int_{\R^N}\frac{|j(x,z) - j(y,z)|^2}{|x-y|^2}+\int_{\R^N\setminus B}\frac{|j(x,z) - j(y,z)|}{|x-y|}
\mu(dz)\le {\bar c}.
\end{align*}
\end{itemize}
\smallskip
The non-linearity $F$ satisfies the following classical
  assumptions
\smallskip
\begin{itemize}
\item [(A2)]  There exists $ \lamb0 >0$ such that for any $x\in \R^d
$, $u,v\in \R$, $p\in \R^d
$, $X\in \SN$ and $l\in \R$,
$$ 
 F(x,u,p,X,l)- F(x,v,p, X,l)\ge  \lamb0 (u-v) \quad \mbox{when} \quad u \ge v.
$$
\item [(A3-1)] $F$ is continuous, and for any $R>0$, there exist moduli of continuity $\omega,
  \omega_R$ such that, for any $|x|,|y|\le R$, $|v|\le R$, $l\in\R$
  and for any $X,Y \in \SN$ satisfying
\begin{equation}\label{ineqmat}
\left[\begin{array}{cc}X&0\\0&-Y\end{array}\right]
 \le \frac1\eps \left[\begin{array}{cc}I&-I\\-I&I\end{array}\right]
+ r(\beta) \left[\begin{array}{cc}I&0\\0&I\end{array}\right]
\end{equation}
for some $\eps >0$ and $r(\beta) \to 0$ as $\beta \to 0$, then, if
$s_i(\beta) \to 0$ as $\beta \to 0$ for $i=1$ and $2$, we have
\begin{align}
\label{cond:a31}
\begin{split}
& F(y,v,\eps^{-1} (x-y)+s_1(\beta), Y,l) - F(x,v,\eps^{-1} (x-y) +
  s_2(\beta),X,l)\\ 
& \le \omega (\beta)+\omega_R (|x-y|+\eps^{-1}|x-y|^2).
\end{split}
\end{align}
\item [(A3-2)] For any $R>0$, $F$ is uniformly continuous on 
$\R^n \times [-R,R] \times B_R \times D_R \times \R$ where 
$D_R := \{ X \in \SN;\ |X| \leq R\}$ and there exist a modulus 
of continuity $\omega_R$ such that, for any $x,y \in \R^d$, 
$|v|\le R$,$ l\in\R$ and for any $X,Y \in \SN$ satisfying (\ref{ineqmat}) and
$\eps >0$, we have
\begin{equation}\label{cond:a32}
F(y,v,\eps^{-1} (x-y), Y,l) - F(x,v,\eps^{-1} (x-y),X,l) \le 
\omega_R (\eps^{-1} |x-y|^2 + |x-y|).
\end{equation}
\item [(A4)] $F(x,u,p,X,l)$ is nondecreasing and Lipschitz continuous
  in $l$, uniformly with respect to all the other variables.
\medskip
\item [(A5)] $M_F:=\sup_{x\in\Omega}|F(x,0,0,0,0)|<\infty.$
\end{itemize}

\medskip

Assumption (A3-1) and (A3-2) are two versions of assumption (3.14) in
the Users' Guide \cite{cil} for possibly unbounded domains. These
assumptions along with (A4) imply that equation \eqref{E} is
degenerate elliptic. Assumptions (A3-1) allows more general
$x$-dependence in the equation (e.g. HJB equations with
 at most linear growth in the derivatives and general $x$-depending coefficients), while (A3-2) allows more general gradient dependence in
the equation (e.g. HJB equations with coefficients which are bounded in $x$ but possibly with $x$-independent superlinear gradient terms).

To be more explicit, consider the linear equation (\ref{example}). The above assumptions hold
if $a \geq 0$, $A(x)= \sigma(x)\sigma^T(x)$ for some matrix
$\sigma$, and, for (A2), $\lambda(x) \geq \lambda_0 >0$ in $\R^N$. Assumptions 
  (A3-1) and (A3-2) are the following two variants of conditions on
  $\sigma, b, \lambda, f$: (A3-1) is satisfied if
  $\sigma$ and $b$ are bounded and {\em locally} Lipschitz continuous
  and $\lambda$ and $f$ 
  are continuous. For (A3-2), $\sigma$, $b$ can have
  a linear growth but one needs the {\em global} Lipschitz continuity
  of $\sigma$ and $b$, and the uniform continuity of $\lambda$ and $f$.

In the local case with no 
 $\I$-dependence in the equation,  assumptions (A2) 
-- (A5) imply comparison, uniqueness, and existence (via Perron's
method) of a bounded viscosity solution of \eqref{E}--\eqref{BC}, cf. e.g.
\cite{cil}. In the nonlocal case when $\Omega=\R^N$ (and no Neumann conditions,
$\Omega^c=\emptyset$), we have 
the following rather classical result which we will need later.
\begin{proposition}[Results for $\Omega=\R^N$]
\label{propRN}
Assume $\Omega=\R^N$ and (A1), (A2), (A4) hold along with either
(A3-1) or (A3-2). 
\medskip

\noindent 
(a) If $u$ and $v$ are respectively an usc bounded above
subsolution and a lsc bounded below supersolution of \eqref{E}
in $\Omega=\R^N$, then $u\leq v$ in $\R^N$.

\medskip
\noindent 
(b) Assume also (A5) holds, then there exists a
unique bounded viscosity solution $u$ of \eqref{E} in $\Omega=\R^N$
satisfying
\begin{align}
\label{bnd_soln}
|u(x)|\leq \frac{M_F}{ \lamb0}\quad\text{in}\quad \R^N.
\end{align}
\end{proposition}
 
Part (a) was proved in \cite{BI} (see Section 5), and Part (b) follows from 
part (a) and Perron's method since $M_F/ \lamb0$ and $-M_F/\lamb0$ are
super and subsolutions of \eqref{E}. Similar results have been given
e.g. in \cite{BBP,pha,JK,BCI}.

Now we come to the first main result of this paper, a comparison
result for the boundary value problem \eqref{E}--\eqref{BC1}.

\begin{theorem}[Comparison I]
\label{compprinciple}
Assume (A1), (A2), (A4) hold along with either (A3-1) or (A3-2). If
$u$ and $v$ are respectively  an usc bounded above subsolution and a
lsc bounded below supersolution  of \eqref{E}--\eqref{BC1}, then
$u\le v$ in $\R^N$. 
\end{theorem}

Uniqueness of solutions follow, and since $\pm \frac{M_F}{\lamb0}$ are
sub/super solutions of \eqref{E} when (A5) holds, we also get $L^\infty$-bounds.

\begin{corollary}  Assume (A1), (A2), (A4) hold along with either (A3-1) or (A3-2).\!\!\!\!\!
\medskip

\noindent(a) There is not more than one bounded solution of \eqref{E}.
\medskip

\noindent (b) If also (A5) holds, then any solution $u$ of \eqref{E} satisfies
\eqref{bnd_soln}. 
\end{corollary}

\begin{remark}
\label{alt_assumptions}
Under assumption (A3-1), the above results also holds if assumption
(A1) is replaced by the much more general assumption: 
\begin{itemize}
\item [(A1-2)] 
Assumption (\ref{Levy}) holds and there exists a constant $\bar{c}>0$ such that
\begin{equation*}
\int_{B}|j(x,z) - j(y,z)|^2
\mu(dz)\le {\bar c}|x-y|^2.
\end{equation*}
\end{itemize}
The proof in Section 5 in \cite{BI} can be modified easily to cover
this case by a clever trick which can be found e.g. in Section 6 in
\cite{JK}. Compared to assumption (A1), assumption (A1-2) allows
more general dependences of $j(x,z)$ in $x$. If we also relax
  \eqref{Levy} so that the constant $c(j)$ is finite only for compact
  subsets of $x\in\Omega$, then the above results also cover the case when
  $j$ has linear growth in $x$.
\end{remark}

 \begin{figure}[h!]
 \includegraphics[scale=1.00]{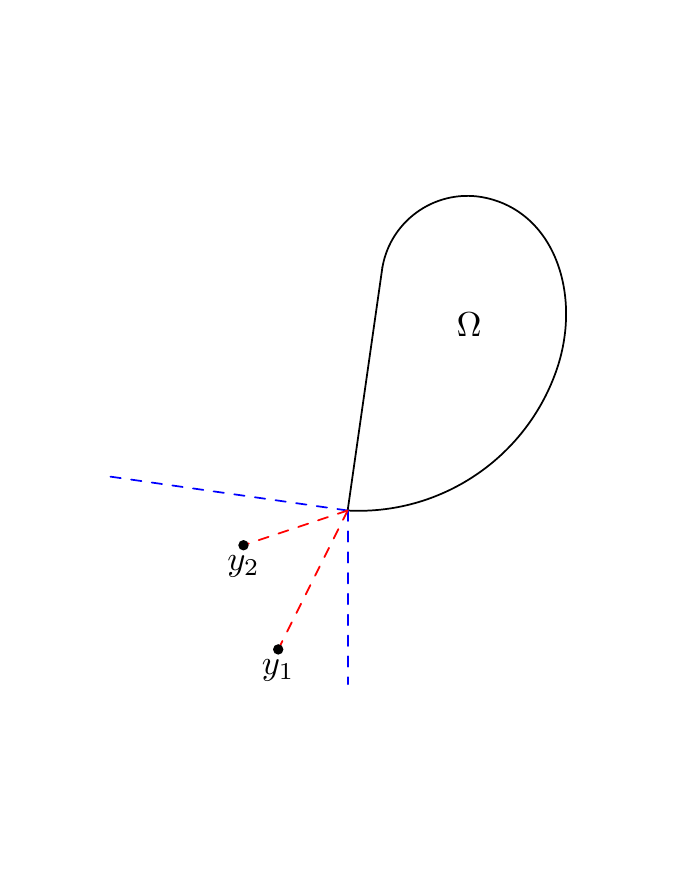}
 \caption{The curves of (BC1) with $\gamma=n$ starting at $y_1$ and
   $y_2$ and ending at a corner point of $\del\Omega$.}
 \end{figure}

\begin{proof}[Proof of Theorem \ref{compprinciple}] 
We introduce the following initial value problem (cf. (BC2)), 
\begin{equation}
\label{caracteristics}
 \dot X_y (t) = -n(X_y(t))\ \ \text{for}\ \ t>0,\qquad   X_y(0)=y.
\end{equation}
Note that the projection on the closed convex set $\omegb$,
$P:\Omega^c\ra\del\Omega$, is also given by 
$$P_y=X_y(\tau_y) \quad \hbox{for any  }y\in\omegb^c \; ,$$
where we recall that
$$\tau_y=\inf\{t>0:X_y(t)\in\domeg\}.$$ 
Since $\Omega$ is convex and $|n|=1$,
it follows that $\{X_y(\cdot)\}_y$ defines a family of constant speed,
finite length, and non-intersecting paths in $\omegb^c$ having the form
\begin{equation}
\label{straightline}
X_y(t)=y-tn(y)\quad\text{for}\quad t\in[0,\tau_y].
\end{equation}
Obviously $\tau_y<\infty$ for all $y$ so that (BC2) is  trivially
satisfied when $\gamma=n$.

We argue by contradiction assuming that 
$$\displaystyle{M:=\sup_{\R^N} \{u(x)-v(x)\}>0}.$$ 
Since $n$ satisfies (BC1'), Lemma~\ref{rem22} applies with $g\equiv0$ and we 
find that $u(x)-v(x)\leq u(P_x)-v(P_x)$ in $\omegb^c$, and hence that $\displaystyle{M =
  \max_{\omegb} \{u(x)-v(x)\}}$. 

Let $\chi : \R \to \R$ be a bounded smooth function such that $\chi(t)
\equiv 0$ for $t\leq 0$, $\chi'(t) >0$ for $t>0$, and $\chi(t) \geq
2(||u||_\infty + ||v||_\infty)$ for $t \geq 1$. 
We double the variables, introducing the function
\begin{align*}
\Psi(x,y) =u(x)-v(y)-\frac{|x-y|^2}
{\epsi^2}-\chi(\beta(|x-x_0|^2+1))-\chi(\beta(|y-x_0|^2+1)),
\end{align*} 
where $\eps,\beta>0$, and $x_0$ is any given point in $\Omega$. It is
easy to see that, for $\beta$ small enough,
$M_{\eps,\beta}=\max_{\R^{2N}} \Psi(x,y)$ exists 
and is attained at some point $(\bar x,\bar y)\in\R^N\times\R^N$ (that depends on $\eps$ and $\beta$).
The crucial and new step in the proof is to show that $(\bar x,\bar
y)\in \bar \Omega\times \bar \Omega$. If this was not the case, then
two applications of Lemma~\ref{rem22} yields that
\begin{equation}\label{ineq-uv}
u(\bar x)-v(\bar y)\leq u( P_{\bar x}) - v( P_{\bar y}).
\end{equation}
Moreover, since $\Omega$ is
convex and $x_0\in\Omega$,  
\begin{equation}\label{ineqconvex}
|\bar x-\bar y|\geq |P_{\bar x}-P_{\bar y}|,\quad
|\bx-x_0|\geq |P_{\bar x}-x_0|, \quad\text{and}\quad|\by-x_0|\geq |P_{\bar
  y}-x_0|,
  \end{equation}
 and then, for $\beta$ small enough, we have the contradiction
\begin{equation}\label{ineq-psi}
M_{\eps,\beta}=\Psi(\bar x,\bar y) < \Psi(P_{\bar x},P_{\bar y}).
\end{equation}

Since $\bar x,\bar y \in \omegb$, the rest of the proof follows
classical arguments.  
Assume $\bx\in \partial\Omega$ and let 
 $$\phi(x,y) = \frac{|x-y|^2}
 {\epsi^2}+\chi(\beta(|x-x_0|^2+1))+\chi(\beta(|y-x_0|^2+1)).$$ 
Note that 
by convexity of $\Omega$, 
$$(\bx-y)\cdot n\geq0\quad\text{for all}\quad y\in\omegb,\ n\in
N_\Omega(\bx).$$
Moreover, this inequality is strict if $y\in\Omega$. Finally, since
$\chi'(t) >0$ for $t>0$, we use the fact that $x_0\in\Omega$ to find that
\begin{align}\label{avoid}
\begin{split}
D_x\phi(\bx,\by) \cdot n &= \frac 2 {\epsi^2}(\bx-\by)\cdot
  n+ \chi' (\beta(|\bx-x_0|^2+1)) 2\beta(\bx-x_0)\cdot  n\\
&> 0\qquad\text{for all}\qquad n\in N_\Omega(\bx) \ \text{and}\ \beta>0.
\end{split}
\end{align}
Therefore, from Definition
\ref{def_vsol}, the equation has to hold at $\bx$,
i.e. $\mathcal F[u,\phi(\cdot,\by)](\bx)\leq 0$. A similar 
argument shows that $\mathcal F[v,-\phi(\bx,\cdot)](\by)\geq 0$ if
$\by\in\del\Omega$. 

Now we are in the situation that  $\bx,\by\in \omegb$ and that the
equation is satisfied at these points. The conclusion of the proof is
then exactly as for the $\R^N$ case, and we omit the standard
details. Under the present assumptions, essentially all the remaining
details can be found in Section 5 in \cite{BI}. But see also
\cite{BBP,JK,pha} for very similar results.   
 \end{proof}

\begin{remark} The key ingredients of the above proof are
\begin{itemize}
\item[(i)] Inequality \eqref{ineq-uv} that comes from
  Lemma~\ref{backtoboundary} or \ref{rem22} and that allow us to compare
  values of $u$ and $v$ outside $\omegb$ with those on inside $\omegb$.
\smallskip
\item[(ii)] Inequality \eqref{ineq-psi} that comes from convexity
  and contraction properties (see \eqref{ineqconvex}). In the above
  proof, the contraction property of the projection on the closed,
  convex set $\omegb$ was playing the key role (allowing us
  to use a very simple test function), but in general the contraction property
  comes from the control on the $X_y$ trajectories w.r.t. $y$.
\smallskip
\item[(iii)] As in the classical Neumann/oblique derivatives boundary
conditions cases, the test-function has to be build in order to allow
us to ``avoid'' the boundary condition (cf. \eqref{avoid}). 
\end{itemize}
These three ingredients are the same in any proof but with different
arguments to handle them. We are going to focus on these arguments. 
\end{remark}


\begin{remark}
 If $\Omega$ is bounded, we can relax assumptions (A3-1) and (A3-2) in
 the standard way and the comparison result will still hold. E.g. since we
 no longer need to prevent maximum points from escaping to 
infinity,  we can set all functions $r,s_1,s_2$ and $\omega$ equal
zero in (A3-1).
\end{remark}

\section{General oblique derivative conditions in non-convex smooth domains}\label{gen-obl}

In this section we consider the general oblique derivative
problem of the form \eqref{E}--\eqref{BC} on a bounded,
possibly non-convex , $C^2$-domain $\Omega$. Compared to section
\ref{sec:results}, the domain and  boundary condition and  are more
general, but the class of equations (see below) and the boundary
regularity are more restricted. 

 Assuming that \eqref{Levy} and (BC1) hold, and we now have the
  following definition of viscosity solutions

\begin{definition}
\label{def_vsol2}
(i) A locally bounded, usc function $u:\R^N\to \R$ is a {\em viscosity subsolution} of
\eqref{E}--\eqref{BC} if it satisfies \eqref{integ}, and for any test
function $\phi\in C^2(\R^N)$ and for any maximum point $x_0\in\R^N$ of
$u-\phi$ in $B_{c(j)\delta}(x_0)$ where $c(j)$ is defined in \eqref{Levy},
\begin{align*}
\begin{cases}
\mathcal{F}
[u,\phi]
(x_0) \le 0& \hbox{if}\quad
x_0\in \Omega,\\[0.2cm]
\displaystyle\min\bigg(\mathcal{F}
[u,\phi]
(x_0), D\phi(x_0)\cdot \gamma(x_0)-g(x_0)\bigg)\le 0 &\hbox{if}\quad
x_0\in \partial \Omega,\\[0.2cm]
D\phi(x_0)\cdot \gamma(x_0)\le g(x_0) & \hbox{if}\quad x_0\in \omegb^c
\end{cases}
\end{align*}
(ii) A locally bounded, lsc function $v:\R^N\to \R$  is a {\em viscosity supersolution} of
\eqref{E}--\eqref{BC} if  it satisfies \eqref{integ}, and for any
test function $\phi\in C^2(\R^N)$ and for any minimum point $x_0\in
\R^N$ of the function $u-\phi$ in  $B_{c(j)\delta}(x_0)$ where $c(j)$
is defined in \eqref{Levy},
\begin{align*}
\begin{cases}
\mathcal{F}
[u,\phi]
(x_0) \ge 0& \hbox{if}\quad
x_0\in \Omega,\\[0.2cm]
\displaystyle \max\bigg(\mathcal{F}
[u,\phi]
(x_0),D\phi(x_0)\cdot \gamma(x_0)-g(x_0)\bigg)\ge 0 &\hbox{if}\quad
x_0\in \partial \Omega,\\[0.2cm]
D\phi(x_0)\cdot \gamma(x_0)\ge g(x_0) & \hbox{if}\quad x_0\in \omegb^c
\end{cases}
\end{align*}
(iii)  A {\em viscosity solution} $u$ of \eqref{E}--\eqref{BC} is a
locally bounded function whose upper and lower semicontinuous envelopes are
respectively sub- and supersolution of the problem.
\end{definition}

To handle non-convex domains and more general boundary conditions, we
will use a  rather complicated test-function which is no longer only a
function of $x-y$ plus small terms. For the proofs to work out we
therefore need to replace assumption (A3-1) and (A3-2) by a more
restrictive assumption similar to the one used in the local case
\cite{ba2}
\medskip

\begin{itemize}
\item [(A3-3)] For any $R,K>0$, there exist moduli of continuity
  $m_{R,K}$ such that, for any $x,y \in \omegb $, $|u|\le R$, $p,q\in
  \R^N$, $l\in\R$, and matrices $X, Y \in \SN$ satisfying
\begin{equation*}
\barre x-y \barre \leq  \eta \varepsilon,\qquad
\barre p-q \barre \leq K \eta \varepsilon (1 + \barre p \barre \wedge
\barre q 
\barre), \quad\text{and}
\end{equation*}
\begin{equation*}
-\frac{K}{\varepsilon^2 } Id
\leq \left(\begin{array}{cc} X & 0 \\ 0 & -Y
\end{array}\right) \leq \frac{K}{\varepsilon^2} 
\left(\begin{array}{cc} Id
& -Id \\ -Id & Id
\end{array}\right) +  K\eta Id ,
\end{equation*}
we have that
$$ F(y,u,q,Y,l)-F(x,u,p,X,l)\leq m_{R,K}
 \left(\eta + |x-y|(1+\barre p \barre \vee \barre q\barre)
+\frac{|x-y|^2}{\varepsilon^2}\right).$$
\end{itemize}

We have the following comparison result.
\begin{theorem}[Comparison II]
\label{compprinciple2}
Assume (A1), (A2), (A3-3), (A4), (BC1), and (BC2) hold. If 
$u$ and $v$ are respectively a locally bounded usc
subsolution and a locally  bounded lsc supersolution  of
\eqref{E}--\eqref{BC}, then $u\le v$ in $\R^N$. 
\end{theorem}

This result will be proved in the subsections below. We start by
introducing the test function we need for the proof.

\subsection{The test-function}\label{testfunc}
\label{tesfunc}

As for local oblique derivative boundary conditions (see
e.g. \cite{ba2} and references therein), the proof of our comparison
result requires  a rather complicated 
test-function. Fortunately there are no major differences 
between the test-function for the local and nonlocal cases, and we now
recall a few facts about the test-function of \cite{ba2} and describe
the adaptations we need to make here. 

We start by changing our definition of the ``distance to the
boundary'' $d$. Now $d$ will be a bounded $C^2$ function which is
equal to the signed distance function 
to $\domeg$  in a
neighborhood of $\domeg$ ($d>0$ in $\Omega$ and  $d<0$ in
$\omegb^c$) and where $n(x):=-Dd(x)\neq0$ in $\Omega^c$. Note
that $n(x)$ is the outward unit normal vector to $\domeg$ for any $x
\in \domeg$.  
The test-function $\pse\in C^2(\R^{2N})$ of \cite{ba2} can then be defined as follows,
\begin{align}
\label{tfunc}
\pse (x,y) = &\ e^{- K_{1} [d(x) + d(y)]}\,
{{|x-y|^2}\over{\varepsilon^2}}\nonumber\\
&-C_{\eta \varepsilon} \Big(\frac{x+y}{2},
e^{- K_{1} [d(x) + d(y)]}\,
{{2(x-y)}\over{\varepsilon^2}}
\Big) 
\big(d(x)-d(y)\big)\\
&+  e^{- K_{1} [d(x) + d(y)]}\,
{{A\big(d(x)-d(y)\big)^2}\over{\varepsilon^2}}-  K_{2} \eta
\varepsilon \,[d(x) + d(y)] \; , \nonumber
\end{align}
for parameters $\eta, \varepsilon >0$ (small), constants $A,
K_{1},  K_{2}$  (large), and where the function
$C_{\eta \varepsilon}$ (see \cite{ba2} page 214) is a suitable smooth
approximation of a bounded 
Lipschitz extension of the solution $t=C(x,p)$ of the equation  
$$ \gamma(x)\cdot(p+tn(x)) - g(x) = 0\qquad\text{for}\qquad p\in\R^N,\ x \text{ near }\del\Omega.$$

The key properties of the test-function are given in the Lemma below.

\begin{lemma}\label{ftch} Assume (BC1) and let 
    $R>0$. If
  $\eta,\varepsilon>0$ are small  
enough, then for $A, K_1,K_2$ large enough, then the function $\pse$
defined in \eqref{tfunc} has the following properties 

(i) For any $x,y\in \R^N$, 
\begin{equation}\label{pos}
\pse(x,y)\geq  K^{-1}{{|x-y|^2}\over{\varepsilon^2}}-K\eps^2-  K_{2} \eta\varepsilon \,[d(x) + d(y)].
\end{equation}

(ii) For $\eps,\eta\in(0,1)$ and $\norm{x-y}\leq \eta 
\varepsilon$, 
\begin{align}
&\norm{D_{x} 
\pse(x,y)}+\norm{D_{y} \pse(x,y)}\geq -K +
K^{-1}{{|x-y|}\over{\varepsilon^2}},\nonumber\\
& \norm{D_{x} 
\pse(x,y)}+\norm{D_{y} \pse(x,y)}  \leq \nonumber\\
&\quad C
{{|x-y|}\over{\varepsilon^2}} + C\Big(1 +\eta^2K_1 e^{2K_1 \|d\|_\infty}+\eps\eta K_2\Big),\label{estp1}\\
&\norm{D_{x} 
\pse(x,y) + D_{y} \pse(x,y)}\leq K 
{{|x-y|^2}\over{\varepsilon^2}}+ K 
(\eta \varepsilon + \varepsilon^2), \text{ and} 
\nonumber\\
&\frac{K}{\varepsilon^2 } Id \leq D^2 \pse(x,y) \leq 
 \frac{K}{\varepsilon^2} \left(\begin{array}{cc} Id 
& -Id \\ -Id & Id 
\end{array}\right) + K\eta Id.
\end{align}

(iii) There is $\delta >0$ such that for  $\norm{x-y} \leq \delta$ and $x,y$ in a neighborhood of $\domeg$,
\begin{align}\label{passssln}
\gamma(x)\cdot D_{x} \pse(x,y))
& > g(x) \quad \hbox{if } d(x) \leq d(y) \; ,\\ 
\label{passursln}
- \gamma(y)\cdot D_{y} \pse(x,y)) &< g(y) \quad \hbox{if } d(y) \leq
d(x) \;,
\end{align}

and if in addition $|x-y|\leq \eta\eps$, then
\begin{align}\label{gammacontrol}
\begin{split}
&- \gamma(x)\cdot \big (D_{x} \pse(x,y) + D_{y} \pse(x,y)\big )\\
&\le -K_1\frac{\nu}{4}
e^{- K_{1} [d(x) + d(y)]}{{|x-y|^2}\over{\varepsilon^2}}- K_2\frac{\nu}{4}
\eta\varepsilon.
\end{split}
\end{align}
\end{lemma}
Except for \eqref{gammacontrol}, these estimates
have essentially been proved in Section 5 in \cite{ba2}. Some new features that 
only marginally changes the proofs are: (i) $x,y$ can now belong to
$\Omega^c$, (ii) inequality \eqref{pos} is slightly more accurate, and (iii) inequalities \eqref{passssln} and \eqref{passursln} are now
given in a neighborhood and not only at  $\del\Omega$.
Moreover, the constants $K$ will in general depend on
$K_1$ and $K_2$, and the
precise dependence is not important except for the term
\eqref{estp1}. The importance of this dependence is both new and
central to this paper (cf. the proof of Lemma \ref{kl2} a)). We will 
therefore prove both \eqref{estp1} and \eqref{gammacontrol} here.

\begin{proof}[Proof of \eqref{estp1} and \eqref{gammacontrol}]
To simplify the computations, we write $\pse$ in the following way
$$ \pse(x,y) = \chi(x-y, d(x)-d(y),\frac{x+y}{2},d(x)+d(y))\; ,$$
where 
\begin{align*} 
&\chi(X,Y,Z,T):= e^{- K_{1} T}
{{X^2}\over{\varepsilon^2}} -C_{\eta \varepsilon} \left (Z,
e^{- K_{1} T}
{{2X}\over{\varepsilon^2}}
\right) 
Y
+e^{- K_{1} T}  {{AY^2}\over{\varepsilon^2}}-  K_{2} \eta \varepsilon T \; .
\end{align*}
In this notation,
\begin{align*}
&D_x \pse (x,y)= \chi_X-\chi_Y n(x) + \frac12\chi_Z -
\chi_Tn(x),\\
&D_x \pse (x,y)+ D_y \pse (x,y) = -\chi_Y (n(x)-n(y)) + \chi_Z -
\chi_T(n(x)+n(y)).
\end{align*}
By the assumptions on $\gamma$ and $g$ and the construction
of $C_{\eta\eps}=C_{\eta\eps}(x,p)$ in \cite{ba2}, there is a $C>0$ such that
$$|C_{\eps\eta}|+|D_xC_{\eps\eta}|\leq C(1+|p|)\qquad\text{and}\qquad
|D_pC_{\eps\eta}|\leq C. $$
Hence there are constants $C_1$ and
$C_2$  such that 
\begin{align*}
|\chi_X|+|\chi_Y|&\leq C_1 + C_2 e^{- K_{1} T}
\Big({{2|X|}\over{\varepsilon^2}} + {{2(1+A)|Y|}\over{\varepsilon^2}} \Big),\\
|\chi_Z| &\leq \Big(C_1 + C_2e^{- K_{1} T} 
{{2|X|}\over{\varepsilon^2}}\Big)|Y|,\\
  |\chi_T| &\leq K_1 e^{- K_{1} T}C_2\Big({{X^2}\over{\varepsilon^2}} + {{(1+A)Y^2}\over{\varepsilon^2}}\Big)+ K_2\eta\varepsilon.
\end{align*}
Since $|X|,|Y|\leq C|x-y|$, estimate \eqref{estp1} now follows.

To prove \eqref{gammacontrol}, we note that by using Cauchy-Schwarz inequality on the $D_pC_{\eps\eta}$-term and taking $A$
large enough,
\begin{align*}
 \chi_T & = -K_{1} e^{- K_{1} T} \left[
{{X^2}\over{\varepsilon^2}} - D_p C_{\eta \varepsilon} \cdot {{2X}\over{\varepsilon^2}}Y
+ {{AY^2}\over{\varepsilon^2}} \right ]
-  K_{2} \eta \varepsilon \\
&\leq-\frac{K_1}{2} e^{- K_{1}
  T}\Big({{X^2}\over{\varepsilon^2}} +
{{AY^2}\over{\varepsilon^2}}\Big)- K_2\eta\varepsilon.
\end{align*}
Let $\mathcal{W}=\{x:\dist(x,\del\Omega)<r\}$, and let $r>0$
be so small that $\gamma\cdot n\geq\frac\nu2$ in $\mathcal{W}$. Such a
set exists by (BC1) and continuity of $\gamma$ and $n$. After an
easy computation based on the above
estimates, the Lipschitz continuity of $n$ ($|n(x)-n(y)|\sim|X|$), the
inequality $\gamma\cdot n\geq\frac\nu2$, 
Cauchy-Schwarz inequality, and finally, taking $K_{1},
K_{2}$ large enough  so that the $\chi_T$-term
dominates, we conclude that \eqref{gammacontrol} 
holds in  $\mathcal{W}$.
\end{proof}

The next lemma plays a key role in the comparison proof.

\begin{lemma}\label{kl2} Assume (BC1) and (BC2), let $\tau_x$ be
  defined in Lemma \ref{backtoboundary}, and $\tau:=\min(\tau_x, \tau_y)$.
\medskip

\noindent (a) For any $\tilde K\geq0$, there are constants
$K_1, K_2$ large enough, such that for any $\eps,\eta>0$ small enough, if
$x,y\in\Omega^c$ are close enough to $\domeg$ and $\norm{x-y}\leq \eta  
\varepsilon/2$, then 
\begin{align} 
\label{g1}
\pse (X_x(\tau),X_y(\tau)) & \leq \pse (x,y) - \tilde K \tau \eta
\eps.
\end{align}

\noindent (b) For any $\eta >0$, there are constants
$K_1, K_2$ large enough, such that for any $\eps>0$, if
$x,y\in\Omega^c$ are close enough to $\domeg$ and $\tau_y\leq\tau_x$,
then  
\begin{align} 
\label{gtransportint}
\pse (X_x(\tau_x),X_y(\tau))& \le \pse (X_x(\tau),X_y(\tau)) -
\int_{\tau_y}^{\tau_x} g(X_x(t)) dt.
\end{align}

\noindent (c) For any $\eta >0$, there are constants
$K_1, K_2$ large enough, such that for any $\eps>0$, if
$x\in\Omega^c$ and $y\in\omegb$ are close enough to $\domeg$, then 
\begin{align*} 
\pse (X_x(\tau_x),y)& \le \pse (x,y) -
\int_{0}^{\tau_x} g(X_x(t)) dt.
\end{align*}
\end{lemma}

\begin{figure}[h!]
 \includegraphics[scale=0.90]{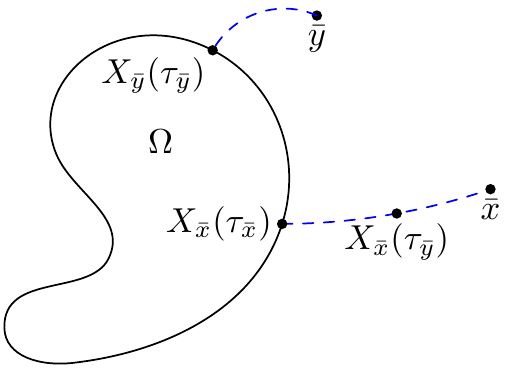}
\caption{Curves of (BC1) with different starting points in the oblique
  case.}
\end{figure}

\begin{proof}
Consider a neighborhood of $\del\Omega$, $\mathcal{W}_r=\{x:\dist(x,\del\Omega)<r\}$, and let $r>0$
be so small that \eqref {gammacontrol} holds, $d(x)=\pm\dist(x,\del\Omega)$, and
$\gamma\cdot n\geq\frac\nu2$ in $\mathcal{W}_r$. Such a set exists by
the definition of $d$, (BC1), and continuity of $\gamma$ and $n$. In the set
$\mathcal{W}_r\cap \Omega^c$ the 
distance to boundary $f(t)=\dist(X_x(t),\del\Omega)=-d(X_x(t))$ is decreasing,
\begin{align}
\label{f}\dot f (t)=-Dd(X_x(t))\cdot \dot X_x(t)=
n(X_x(t))\cdot(-\gamma(X_x(t)))<-\frac\nu2, 
\end{align}
and hence $X_x(t)\in \mathcal{W}_r\cap \Omega^c$ for all $t\in[0,\tau_x]$ and
$x\in\mathcal{W}_r\cap \Omega^c$. 

Next we note that if $L$ is the Lipschitz constant of  $\gamma$, then by Gr\"onwall's inequality, 
\begin{equation}\label{est-traj}
|X_x(t)-X_y(t)| \leq e^{Lt}|x-y|.
\end{equation}
We estimate $\tau_x$, and hence also $\tau_y$ and $\tau$, by
integrating \eqref{f} from $t$ to $\tau_x$ and noting that $f(\tau_x)=0$
$$\frac\nu2(\tau_x-t)<
f(t)=\dist(X_x(t),\del\Omega)\leq
\dist(x,\del\Omega)\quad\text{for}\quad t\in[0,\tau_x].$$
Hence if $r$ is small, $\tau$ will also be small in $\mathcal{W}_r\cap \Omega^c$. In the rest of the
proof we take  $x,y\in\mathcal{W}_r\cap \Omega^c$, and then we take $r$ so small
that also $|X_x(t)-X_y(t)|\leq \eta \varepsilon$ for all $t \in [0, \tau]$
and $x,y\in\mathcal{W}_r\cap \Omega^c$ such that $|x-y|\leq\frac{\eta\eps}2$.

We now prove part (a). We start by using the definition of
$X_x(t)$ (see (BC2)) to show that
\begin{align*}
&\frac{d}{dt}\left[\pse (X_x(t),X_y(t))\right]
= - D_x \pse \cdot \gamma(X_x(t)) - D_y \pse \cdot \gamma(X_y(t))\\
&=
- [D_x \pse + D_y \pse] \cdot \gamma(X_x(t)) - D_y \pse \cdot [\gamma(X_y(t))-\gamma(X_x(t))].
\end{align*}
We may use \eqref{estp1} (check!) and the Lipschitz
continuity of $\gamma$ to have
\begin{align*}
&|D_y \pse \cdot [\gamma(X_y(t))-\gamma(X_x(t))]|\\
 &\le  L |X_x(t)-X_y(t)|\cdot C\Big(
{{|X_x(t)-X_y(t)|}\over{\varepsilon^2}} +1+\eta^2 K_1 e^{K_1 2\|d\|_\infty}+\eps\eta K_2\Big)\\
&\leq  L C
\Big({{|X_x(t)-X_y(t)|^2}\over{\varepsilon^2}} +\eta\eps\Big(1+\eps\eta K_2+ \eta^2K_1 e^{2K_1
  \|d\|_\infty}\Big) \Big), 
\end{align*}
and by  \eqref {gammacontrol} we immediatly find that
\begin{align*}
&- [D_x \pse + D_y \pse] \cdot \gamma(X_x(t))\\
& \le  -K_1\frac{\nu}{4}\, e^{- K_{1} [d(X_x(t)) + d(X_y(t))]}\, {\frac {|X_x(t)-X_y(t)|^2}{\varepsilon^2}}- K_2\frac{\nu}{4} \eta\varepsilon.
\end{align*}
Since $\gamma\cdot n\geq\frac\nu2$, we then find that
\begin{align*}
&\frac{d}{dt}\left[\pse (X_x(t),X_y(t))\right]\\
&\leq \Big(  L C-K_1 \frac{\nu}{4} e^{- K_{1} [d(X_x(t)) +
  d(X_y(t))]}\Big) {\frac {|X_x(t)-X_y(t)|^2}{\varepsilon^2}}\\
&\quad+
\Big(L C\Big(1+\eps\eta K_2+ \eta^2K_1 e^{2K_1
  \|d\|_\infty}\Big)-K_2\frac{\nu}{4} \Big) \eta\varepsilon   \\ 
& \leq -\tilde K \eta \eps
\end{align*}
for any given constant $\tilde K$ since we can take first $\eps,\eta$ small
enough and then $K_1$ and finally $K_2 $ as large as we want. The
conclusion follows by integrating  from $0$ to $\tau$.

To prove (b), we notice that $\tau=\tau_y \leq \tau_x$.  Since $\dot
X=-\gamma(X)$ and $d(X_x(t)) \leq d(X_y(\tau))=0$ for $\tau= \tau_y \leq t \leq 
\tau_x$, we can use \eqref{passssln} to find that 
\begin{align*}
&\frac{d}{dt}\left[ \pse (X_x(t),X_y(\tau))\right]=-
D_x \pse \cdot \gamma(X_x(t))\leq -g(X_x(t)).
\end{align*}
Part (b) now follows by integrating from $\tau_y$ to $\tau_x$.
The proof of (c) is just like the proof of (b) replacing $X_y(\tau)$
by $y$ and setting $\tau=0$. 
 \end{proof}

\subsection{Proof of Theorem~\protect\ref{compprinciple2}}
In order to show that $u(x)-v(x)\leq 0$ in $\R^N$, we first notice that, by (BC2) and Lemma
\ref{backtoboundary}, 
$$ u(x)-v(x)\leq u(X_x(\tau_x))-v(X_x(\tau_x))\quad\text{for any}\quad x\in\omegb^c,$$
and hence since $X_x(\tau_x)\in\del\Omega$, it follows that $u-v$ is bounded from above in $\R^N$ and
$$M=\displaystyle{\sup_{\R^N} \{u(x)-v(x)\} =
  \max_{\omegb} \{u(x)-v(x)\}}.$$

In the rest of the proof we argue by contradiction assuming that 
$$M>0.$$ 
Then we define
$$w_\beta(x)= u(x)-v(x)-2\chi(-\beta d(x)),$$ 
where $\beta >0$  (small), $\chi$ is the function we introduced in
the proof of Theorem~\ref{compprinciple}, and $d$ is the signed
distance function to $\domeg$ ($d<0$ in $\omegb^c$). Since the
$\chi$-term vanishes on $\omegb$ and is strictly positive on $\omegb^c$,
$w_\beta$ has maximum points only on $\omegb$ and these points are
also maximum points of $u-v$.  

Now we double the variables introducing the function
$$ \Phi(x,y) = u(x)-v(y) - \pse(x,y)-\chi(-\beta d(x))-\chi(-\beta d(y)).$$
By standard arguments involving the definition of $\chi$ and the
properties of $\pse$ given in Lemma \ref{ftch} (in
particular \eqref{pos}), this
function achieves its maximum at a point $(\bar x,\bar y)\in
\R^N\times\R^N$ (depending on $\eps$, $\eta$ and
$\beta$). Moreover, 
for fixed $\eta$ and $\beta$, 
$$ \frac{|\bar x- \bar y|^2} {\epsi^2} \to 0\qquad\text{as}\qquad \eps\ra0,$$
and $\bar x, \bar y$ converges (along
subsequences) to a maximum point $\tilde x$ of $w_\beta(x)$, i.e. to a
point in $\omegb$. In particular, $\bar x, \bar y$ will be arbitrarily close to
$\del\Omega$ if $\eps$ close enough to $0$.

We will show that $\bar x, \bar y$ are in $\omegb$ when $\eps>0$ is
small enough. 
Again we argue by contradiction assuming that $\bx,\by$ are not both in
$\omegb$. 
Assume e.g. that $\bx,\by\in\Omega^c$ and that $\tau_{\bar y}
\le \tau_{\bar x}$. We will get a contradiction to the maximum point
property by showing that
$$\Phi(\bar x,\bar y)<\Phi(X_{\bar x}(\tau_{\bar x}),X_{\bar y}(\tau_{\bar y})).$$
To do this, we start by using 
Lemma~\ref{backtoboundary} for both $u$ and $v$ to see that
\begin{align*}
\Phi(\bar x,\bar y)\le &\ u(X_{\bar x}(\tau_{\bar x}))-v(X_{\bar
  y}(\tau_{\bar y}))\\
& + \int_0^{\tau_{\bar y}}\Big(g(X_{\bar x}(s))-g(X_{\bar y}(s))\Big) ds +
\int_{\tau_\by}^{\tau_\bx} g(X_\bx(t)) dt\\
& - \pse(\bar x,\bar y)-\chi(-\beta d(\bx))-\chi(-\beta d(\by)).
\end{align*}
But from Lemma \ref{kl2}, using first part (b) and then
part (a),   
\begin{align*}
\int_{\tau_{\by}}^{\tau_\bx} g(X_\bx(t))\, dt &\le \pse
(X_\bx(\tau_{\bar y}),X_\by(\tau_{\bar y}))-\pse
(X_\bx(\tau_{\bx}),X_\by(\tau_{\bar y}))\\
& \leq \pse (\bar x, \bar y) -2\tilde K \tau \eta \eps- \pse (X_\bx(\tau_{\bar x}),X_\by(\tau_{\bar y})),
\end{align*}
and by Lipschitz regularity of $g$ and $\gamma$ and the estimate \eqref{est-traj},
 $$\int_0^{\tau_{\bar y}} \Big(g(X_{\bar x}(s))-g(X_{\bar y}(s))\Big) ds\le
 \tau_{\bar y} L_g e^{L_\gamma \tau_\by}|\bar x-\bar y|.$$
Hence we find that
\begin{align*}
& \Phi(\bar x,\bar y)\\
&\le \Phi(X_{\bar x}(\tau_{\bar x}),X_{\bar
   y}(\tau_{\bar y})) -2\tilde K \tau_{\bar y}\eta\varepsilon +
 \tau_{\bar y} L_g e^{L_\gamma \tau_\by} |\bar x-\bar y| - \chi(-\beta d(\bar
 x))-\chi(-\beta d(\bar y)), 
\end{align*}
and since $|\bar x-\bar y|\le \eta\varepsilon$, we get the
contradiction by choosing $\tilde K$ large enough.
A similar argument covers the case when $\tau_\by\geq\tau_\bx$, and we
can conclude that at least one of $\bar x$ and $\bar y$ belongs
 to $\omegb$.

Next we show that it is not possible that e.g. $\bx\in \omegb^c$ while
$\by\in\omegb$. This time we use 
Lemma~\ref{backtoboundary} for only $u$ to see that
\begin{align*}
\Phi(\bar x,\bar y)\le &\ u(X_{\bar x}(\tau_{\bar x}))-v({\bar y})+\int_{0}^{\tau_\bx} g(X_\bx(t)) dt- \pse(\bar x,\bar y)-\chi(-\beta d(\bx)).
\end{align*}
But by Lemma \ref{kl2}(c),
$$\int_{0}^{\tau_\bx} g(X_\bx(t)) dt - \pse(\bar x,\bar y)\leq
-\pse(X_{\bar x}(\tau_{\bar x}),\bar y),$$
and hence we find again a contradiction
$$\Phi(\bar x,\bar y)\le \Phi(X_{\bar x}(\tau_{\bar x}),\bar
y)-\chi(-\beta d(\bx))<\Phi(X_{\bar x}(\tau_{\bar x}),\bar
y).$$
The case that $\by\in \omegb^c$ while $\bx\in\omegb$ gives a
contraction in a similar way, and in view of previous arguments we can
conclude that $\bx,\by\in\omegb$, at least when $\eps>0$ is small enough. 

Since $\pse$  satisfies by \eqref{passssln} and
\eqref{passursln}, it follows that 
the equation (the sub and supersolution inequalites), and not the
boundary condition, has to hold if $\bx$ or $\by$ 
belongs to $\del\Omega$ and hence for all
$\bx,\by\in\omegb$.  By assumption, $u,v$ are bounded
  on $\omegb$ so that assuption (A3-3) can be applied with
  $R=\max_{\omegb}(|u|+|v|)$. At this point we can conclude the proof as in 
the $\R^N$-case, sending first $\eps\ra0$, then $\eta\ra0$, and finally
$\beta\ra0$.
We omit the standard details only noting that under the present
assumptions, essentially all the remaining 
details can be found in Section 5 in \cite{BI}. But see also 
\cite{JK,BBP,pha} for very similar results.

 \section{Penalization of the domain}\label{asymp}

In this section we show that our way of defining Neumann type boundary
conditions is consistent with the so-called penalization of the domain
method introduced by Lions and Sznitman in \cite{pls}. We extend the
results of \cite{pls} to our non-local setting, proving the
convergence of a sequence of solutions of penalized $\R^N$-problems to
the solution of \eqref{E}. We give separate results in the convex case of
Section \ref{sec:results} and the oblique case of Section \ref{gen-obl}.

\subsection{Neumann conditions on convex domains}
\label{Neu_subsec}
In this section we assume that $\Omega$ is convex and possibly
unbounded. Let $\bar d$ be the distance to $\Omega$ defined in
Section \ref{sec:results} and $n=D\bar d$ in $\omegb^c$. Note that
$\bar d=0$ in $\omegb$  
and define $\tilde d=\min(\bar d,1)$. By the Lipschitz continuity of
$\tilde d$ and the convexity of $\bar d$, the continuous vector field 
$x\mapsto \tilde d(x)n(x)$ (extended by $0$ to $\omegb$) satisfies
\eqref{prop-imp} in $\R^N$. This property will play a key role below.

Moreover, we assume that (A1)--(A5) hold, and if
necessary, we extend the data and $F$ to $\R^N$ in a way that preseverves 
these properties. We study the following equation for the penalization
of the domain, cf. \cite{pls}
\begin{align}
   \label{E_eps}
F(x,u,Du,D^2 u, \I [u](x)) + \frac{1}{\kappa}\tilde d(x)n(x)\cdot Du=0
  \quad \text{ in }  \quad  \mathbb{R}^N.
\end{align}
 where $0<\kappa\ll 1$. Since $\tilde d(x) n(x)$ satisfies \eqref{prop-imp},
Equation~\eqref{E_eps} with $\kappa>0$ fixed satisfies (A1)--(A5) as long as $F$
does.

\begin{theorem}
\label{pendom_conv}
Assume that (A1), (A2), (A4), (A5) hold along with either (A3-1) or
(A3-2). Then the viscosity solution $u_\kappa$ 
of \eqref{E_eps} converge locally uniformly to a bounded
continuous function $u$ which is the viscosity solution of
\eqref{E}--\eqref{BC} according to Definition \ref{def_vsol}.
\end{theorem}

\begin{remark}
This result provides an existence result for
\eqref{E}--\eqref{BC}. In contrast to the more difficult
Dirichlet case in \cite{BCI}, we have existence also when the there is
loss of boundary conditions. 
\end{remark}

We need the following auxilliary result that follows from Proposition \ref{propRN}.
\begin{lemma}
\label{corRN}
Assume that (A1), (A2), (A4), (A5) hold along with either (A3-1) or (A3-2).
Then there exists a
unique bounded viscosity solution $u_\kappa$ of \eqref{E_eps} satisfying
$$|u_\kappa(x)|\leq \frac{M_F}{\gamma}\quad\text{in}\quad \R^N.$$
\end{lemma}

\begin{proof}
Note that $u_\kappa$ is bounded uniformly in $\kappa$, and that we
may rewrite \eqref{E_eps} in the following equivalent way
\begin{align}
   \label{E_eps2}
&G_\kappa(x,u,Du,D^2 u, \I [u](x)) = 0\quad \text{ in }  \quad
\mathbb{R}^N\\ 
\intertext{where}
&G_\kappa(x,r,p,X,l)=\begin{cases}
F(x,r,p,X,l), &\text{for } x\in\bar\Omega,\\
\frac{\kappa}{\bar d(x)} F(x,r,p,X,l)+ n(x)p,&\text{for } x\in\bar\Omega^c.
\end{cases}
\end{align}
Now we introduce the half relaxed limits
$$\underline f(x):={\liminf}_* f_\kappa(x)=\liminf_{
\scriptsize
\begin{array}{l}
 y \ra
    x\\
\kappa\ra0
\end{array}}
f_\kappa(y), \quad\overline f(x):={\limsup}^* f_\kappa(x)=\limsup_{
\scriptsize
\begin{array}{l}
y  \ra
    x\\
\kappa\ra0
\end{array}}
f_\kappa(y). 
$$ 
Note that $\underline{F}=F$ and 
$$\underline{G}(x,r,p,X,l)=\begin{cases}
F(x,r,p,X,l)& \text{when}\quad x\in\Omega,\\
\min\{F(x,r,p,X,l),\inf_{n\in N_\Omega(x)}n\cdot p\}& \text{when}\quad x\in\partial\Omega,\\
n(x)p& \text{when}\quad x\in\bar\Omega^c,
\end{cases}$$
and in a similar way we find that $\overline G$ is like $\underline G$
with $\max$/$\sup$ replacing the $\min$/$\inf$. As a consequence of
the stability of viscosity solutions, see e.g. Theorem 1 in \cite{BI}, 
$\overline u=\limsup^*u_\kappa$ is a viscosity subsolution of 
$$\underline{G}(x,u,Du,D^2u,\I[u])=0\quad\text{in}\quad\R^N,$$
while $\underline u=\limsup_*u_\kappa$ is a viscosity supersolution of
$$\overline{G}(x,u,Du,D^2u,\I[u])=0\quad\text{in}\quad\R^N.$$
By Definition \ref{def_vsol} this means that $\overline u$ and $\underline
u$ are sub- and supersolutions of \eqref{E}--\eqref{BC}, and hence by 
comparison, Theorem \ref{compprinciple},
$$\overline u\leq \underline u.$$
The opposite inequality is true by definition of $\overline u$, and
hence we have $\overline u=\underline u=:u$. It follows that
$u$ is continuous and $u_\kappa\ra u$ locally uniformly,
 as is standard in viscosity solution theory. 
\end{proof}

\subsection{Oblique boundary value problems in bounded smooth domains}\label{sec:5.2}

In this section, we assume as in Section \ref{gen-obl}, that $\Omega$
is a bounded $C^2$ domain. We study the following equation for the penalization
of the domain, cf. \cite{pls}: 
\begin{align}
   \label{E_kappa}
F(x,u_\kappa,Du_\kappa,D^2 u_\kappa, \I [u_\kappa](x)) +
\frac{1}{\kappa}\tilde d(x)[\gamma(x)\cdot Du_\kappa-g]=0
  \quad \text{ in }  \quad  \mathbb{R}^N.
\end{align}
where $0 < \kappa \ll 1$ and  $\tilde d$ is defined as in the previous section.  

We want to prove that we can obtain the oblique boundary value problem
\eqref{E} from the penalized problem \eqref{E_kappa} in the
limit as $\kappa\ra0$. In \eqref{E} (Definition \ref{def_vsol2}), only
$F$'s values at $\bar\Omega$ play any role, and we may modify
equation \eqref{E_kappa} in $\bar\Omega^c$ and still obtain \eqref{E}
from \eqref{E_kappa} in the 
limit as long as (A1)--(A5) still hold.

In order to avoid difficulties related to
comparison results for sub and supersolutions, we assume that
$F(x,u,p,M,l)\equiv \lamb0 u$ for $x$ large enough, say for $|x|\geq \tilde R$, where $\lamb0$ is
given by (A2). Taking into account the fact that the truncation on the
distance function implies that  $\tilde d(x)\equiv 1$ for $x$ large
enough, the equation outside a large enough ball reduces to
$$
\lamb0 u_\kappa + \frac{1}{\kappa}[\gamma(x)\cdot Du_\kappa-g] = 0\; ,
$$
which can be treated by a slight adaptation of the technics used in
Section~\ref{prelim} as we will see it later on. For other
extensions of $F$, additional conditions are typically needed 
to handle the growth (typically linear) of the solutions at infinity.

 Here it is unavoidable to impose additional assumptions on $\gamma,g,j,\mu$ to
  satisfy the integrability assumption (\ref{integ}), i.e. to balance the
  growth $u(x+j(x,\cdot))$ with the decay of $\mu$ at
  infinity for solutions $u$ of \eqref{E} and \eqref{E_kappa}. We are going to use (BC3) and refer the reader to the discussion at the end of Section~\ref{prelim}.

We just recall that, in the case when $g$ has compact support, the
solutions are expected to be bounded by Lemma~\ref{backtoboundary}
and no additional assumption on $j$ and
$\mu$ is needed.  On the contrary,  if, for example, $g\equiv 1$, then
the integral of $g$ in (\ref{maxatbdr}) suggests a linear growth and
one has to impose suitable hypothesis on $\gamma, j$ and $\mu$ in order to
satisfy (\ref{integ}). Moreover, if we were considering more general
extension of $F$, we would need a framework where we can compare sub
and supersolutions with linear growth. Our restrictive extension
allow us to avoid such (useless) technicalities.

%

\begin{theorem}
\label{PenDom2}
Assume that (A1)--(A5) and (BC1)--(BC3) hold. Then, for any $\kappa>0$, there exists a unique continuous viscosity solution $u_\kappa$ of \eqref{E_kappa} which is uniformly locally bounded. Moreover, as $\kappa \to 0$, $u_\kappa$ converges locally uniformly to the unique viscosity solution $u$ of \eqref{E}--\eqref{BC}.
\end{theorem}

In the proof we use the following lemma.

\begin{lemma} \label{theta}
Assume (BC1)-(BC3). There exists a $C^\infty$ function $\theta : \R^N \to \R$ such that
$$ \gamma (x) \cdot D\theta (x) \geq 1\quad \hbox{for $x$ in a neighborhood $\mathcal{W}$ of  }\omegb^c\; .$$ Moreover $\theta $ satisfies 
$$ |\theta (x) | \le \tilde c ( 1 +|x|)\quad \hbox{in  } \R^N ,$$
for some $\tilde c >0$.
\end{lemma}

We prove this result after the proof of Theorem \ref{PenDom2}.

\begin{proof}[Proof of Theorem \ref{PenDom2}]  We just sketch the
  proof of the existence and uniqueness of $u_\kappa$ when $g$ is
  not compactly supported. This case involves the function $\theta$ of
  Lemma~\ref{theta} while the other case is easier and involves a
  similarly defined but bounded function $\theta$
  (where $D\theta\cdot\gamma>1$ only on a compact set). 

The strong comparison principle (and hence uniqueness) for
  \eqref{E_kappa} holds by standard 
  argument and a slight modification of the argument of
  Section~\ref{prelim} that we explain now. If $u$ is a subsolution of
  \eqref{E_kappa} then we have  
$$
\kappa\lamb0 u_\kappa + \gamma(x)\cdot Du_\kappa-g = 0\quad\hbox{in  }\bar B^c_{\tilde R}\; ,
$$
where $\tilde R$ is defined above, $B_{\tilde R}$ is the ball centered at $0$ with the (large) radius $\tilde R$. A slight modification of the arguments of 
Section~\ref{prelim} shows that, if $y\in \bar B^c_{\tilde R}$ and if $X_y(s) \in \bar B^c_{\tilde R}$ for $s\in [0,t)$ then
$$
u_\kappa (y) \le \int_0^t g(X_y(s))\exp(-\kappa\lamb0 s)ds + u_\kappa (X_y(t)) \exp(-\kappa\lamb0 t).
$$
Using this result, we can reduce to the case where the maximum points
are in a fixed compact subsets of $\R^N$ and then classical comparison
arguments apply.

Using Lemma~\ref{theta}
and (A2), it is easy to check that, choosing first $C_2>0$ and then
$C_1$ large enough, $\pm (C_1 +C_2\theta(x))$ are respectively
viscosity super and subsolutions of \eqref{E_kappa}. Then we can apply
Perron's method to obtain the exisitence of a solution $u_\kappa$ such that
$$-(C_1 +C_2\theta(x))\leq u_\kappa (x) \leq C_1 +C_2\theta(x) \quad \hbox{in  } \R^N \; .$$

Since the $u_\kappa$'s are locally uniformly bounded, we can use the
half-relaxed limits method.  We rewrite \eqref{E_kappa} in the
following equivalent way as 
$G_\kappa(x,u,Du,D^2 u, \I [u](x)) = 0$ in $\R^N$ where
\begin{align*}
&G_\kappa(x,r,p,X,l)=\begin{cases}
F(x,r,p,X,l), &\text{for } x\in\bar\Omega,\\
\frac{\kappa}{\bar d(x)} F(x,r,p,X,l)+ \gamma(x)p-g,&\text{for } x\in\bar\Omega^c.
\end{cases}
\end{align*}
As in the proof of Theorem \ref{pendom_conv}, we compute the half relaxed limits
and find that
$$\underline{G}(x,r,p,X,l)=\begin{cases}
F(x,r,p,X,l)& \text{when}\quad x\in\Omega,\\
\min\{F(x,r,p,X,l),\gamma(x)p-g\}& \text{when}\quad x\in\partial\Omega,\\
\gamma(x)\cdot p-g& \text{when}\quad x\in\bar\Omega^c,
\end{cases}$$
and that $\overline G$ is like $\underline G$ with a $\max$ replacing
the $\min$, and we find that $\overline u$ is a
viscosity subsolution of the $\underline{G}(x,u,Du,D^2u,\I[u])=0$ and
while $\underline u$ is a viscosity supersolution of
the equation $\overline{G}(x,u,Du,D^2u,\I[u])=0$  in $R^N$. We
conclude as before that $\overline u=\underline u=:u$ and $u_\kappa\ra u$ locally
uniformly. 
\end{proof}

Now we give the proof of Lemma \ref{theta}.

\begin{proof}[Proof of Lemma \ref{theta}] This is a routine adaptation
  of classical arguments. 
Taking $\delta >0$ small enough and denoting by $D_\delta :=\{x \in \R^N; d(x) \leq \delta\}$ where $d$ is defined in Section~\ref{testfunc}, we can solve the problem
\begin{align}
\label{w-eqn}
\gamma (x) \cdot Dw (x) = 2 \quad \hbox{in  }D_\delta\; ,\quad w=
0\quad \hbox{on  }\partial D_\delta\; .
\end{align}
Indeed, arguing as in Lemma~\ref{backtoboundary} with $g=2$ , we have, for any $y\in D_\delta$
$$w (y)= 2\tau^\delta_y\qquad\text{for}\qquad
\tau^\delta_y=\inf\{t>0;\ X_y(t) \in \partial D_\delta\},$$ 
and the function $w$  is finite (thus well-defined)  because of
(BC2).

We prove that $w$ is locally Lipschitz continous in $\overline
D_\delta$ if $\delta$ is so small that by (BC1),
$$\gamma(x)\cdot n(x)>\frac\nu2\qquad\text{in}\qquad
{\Delta_\delta}=\{x: |d(x)|< \delta\}.$$
We first check that $w$ is Lipschitz continuous in $\ol{\Delta}_\delta$. Let  $x,y\in
\ol{\Delta}_\delta$, $f(t) := d( X_x (t+\tau^\delta_y))$,
and note that if $\tau_x>\tau_y$, then
 $$f'(t) = \dot X_x (t+\tau^\delta_y))\cdot Dd(X_x
 (t+\tau^\delta_y))=\gamma(X_x (t+\tau^\delta_y))\cdot n(X_x
 (t+\tau^\delta_y))$$
for $t\in(0,\tau_x-\tau_y)$.
We integrate from $0$ to $\tau^\delta_x-\tau^\delta_y$ and use (BC1)
to find that
\begin{equation}\label{tauestimate}
 \frac\nu2|\tau^\delta_x-\tau^\delta_y | \le  |d( X_x (\tau^\delta_y))| \le  |X_x(\tau^\delta_y)-X_y(\tau^\delta_y)|,
\end{equation}
where the last inequality is a consequence of the definition of the
distance of the point $ X_x (\tau^\delta_y)$ to the boundary. Then if
 $L$ is the Lipschitz constant of $\gamma$, inequality
 \eqref{est-traj} holds and we may 
 use e.g. (BC3) to obtain that
\begin{equation}\label{taulipsch}
\frac\nu2|\tau^\delta_x-\tau^\delta_y | \le  e^{L\tilde c(1+ R)}|x-y|,
\end{equation}
where $R=\max_{x\in \ol{\Delta}_\delta}|x|$. 
It follows that $w$ is Lipschitz in
$\ol{\Delta}_\delta$.

Let $x,y\in \ol D_\delta\setminus \ol{\Delta}_\delta$ be near
one another and take a $T> 0$ such that $X_x(T)\in \Delta_\delta$.
Such $T$ exists and $T\leq \tilde c(1+|x|)$ by (BC3). By inequality
\eqref{est-traj}, we can (and do)
take $y$ close enough to $x$ so that also $X_y(T)\in
\Delta_\delta$. Then $\tau_x^\delta=T+\tau_{X_x(T)}^\delta$ and
$\tau_y^\delta=T+\tau_{X_y(T)}^\delta$, and hence by (BC3) and inequalities
\eqref{taulipsch} and  \eqref{est-traj},
\begin{equation*}
\frac\nu2|\tau^\delta_x-\tau^\delta_y | \le  e^{L\tilde c(1+
  R)}|X_x(T)-X_y(T)|\leq e^{L\tilde c(1+ R)}e^{L\tilde c(1+ |x|)}|x-y|.
\end{equation*}
This completes the proof of local Lipschitz continuity of $w$.

 The next step is to regularize $w$ through a classical convolution
 argument to obtain the smooth function $\theta$. But since $w$ is
 only locally Lipschitz continuous, we have 
 to regularize locally and use a covering argument to build the global
 regularization of $w$. The covering argument is completely standard
 and will not be detailed here. 
 
 Locally we define $w_\eps(x)=w*\rho_\eps(x)$ for $x\in
D_{\frac{\delta}2}$ where $0<\eps<\frac\delta2$ and
$\rho_\eps(x)$ is the standard mollifier, i.e. a positive $C^\infty$-function with
mass one and support in $|x|<\eps$. By the regularity of $w$, $Dw$
exists a.e. and hence equation \eqref{w-eqn} holds a.e. It
follows that $(Dw\cdot \gamma)*\rho_\eps=2$ in
$D_{\frac{\delta}2}$. By the definition of the convolution and of
$\rho_\eps$, the Lipschitz continuity of $\gamma$, and the local boundedness of $Dw$, we are lead to
\begin{align*}
Dw_\eps\cdot \gamma(x) &= (Dw\cdot \gamma)*\rho_\eps (x)
+\int
Dw(y)\cdot(\gamma(y)-\gamma(x)) \rho_\eps(x-y)\ dy\\
&\geq 2 - \|Dw\|_{L^\infty(B(x,\eps))} L_\gamma
\eps\qquad\text{in}\qquad D_{\frac{\delta}2}.
\end{align*}
Hence for any bounded subset $K\subset D_{\frac\delta2}$ we can take
$\eps=\eps_K$ so small that
$$\gamma\cdot Dw_{\eps_K}\geq 1 \qquad\text{in }\qquad \ol
K. $$

Finally, the bound on $|\theta|$ follows directly from a similar bound
for $w$ and a suitable (local) choice of $\eps$. The bound for $w$ is a direct consequence of Assumption (BC3) and Lemma~\ref{backtoboundary}.
\end{proof}

\end{document}